\documentclass  [a4paper, 11pts] {article}
\usepackage [utf8] {luainputenc}
\usepackage [T1] {fontenc}
\usepackage{csquotes}
\usepackage {lmodern}
\usepackage [english] {babel}
\usepackage{hyperref} 

\usepackage {amsmath}
\usepackage{xypic}
\usepackage {amssymb}
\usepackage{mathabx}

\usepackage{amsthm}
\usepackage{MnSymbol}
\usepackage{pifont}

\usepackage {calrsfs}
\usepackage {stmaryrd}
\usepackage{fancyhdr}
\usepackage[top = 3cm, bottom = 3cm, left = 3cm, right = 3cm]{geometry}

\usepackage[
backend=biber,
style=alphabetic,
sorting=ynt
]{biblatex}
\usepackage{graphicx}
\usepackage{caption}
\usepackage{subcaption}

\newtheorem*{thm*}{Theorem}
\newtheorem*{pro*}{Proposition}
\newtheorem*{lem*}{Lemma}
\newtheorem*{Def*}{Definition}
\newtheorem*{exm*}{Example}
\newtheorem*{cor*}{Corollary}
\newtheorem*{req*}{Remark}
\newtheorem*{con*}{Conjecture}

\newtheorem{thm}{Theorem}[section]
\newtheorem{pro}[thm]{Proposition}
\newtheorem{lem}[thm]{Lemma}

\newtheorem{exm}[thm]{Example}
\newtheorem{cor}[thm]{Corollary}
\newtheorem{req}[thm]{Remark}

\newcommand{\R}{\mathbb{R}}

\newcommand{\Z}{\mathbb{Z}}

\newcommand{\eps}{\varepsilon}
\newcommand{\Out}{{\mathrm{Out}}}
\newcommand{\Aut}{{\mathrm{Aut}}}
\newcommand{\Min}{{\mathrm{Min}}}

\title{Which $F_3$-by-$\Z$s are CAT(0)?}
\author{Leo \textsc{Delage}}
\date{}

\begin{document}

\maketitle

\begin{abstract}

In this note we point out a mistake in theorem 4.4 of \cite{Sam06}, which states that a semidirect product $F_3\rtimes_\phi\Z$ whose defining automorphism $\phi$ is triangular polynomially-growing and fixes a free factor of rank $2$ is a CAT(0) group. We give and prove the corrected statement: such a group is CAT(0), if and only if $\phi$ is the identity or if the element of $F_2$ twisting the non-fixed generator is not in the commutator subgroup of $F_2$. This gives new examples of free-by-cyclic groups that cannot act properly by semisimple isometries on a CAT(0) space, that are similar to {Gersten}'s examples \cite{Ger94}. Some of them, interestingly, are still hierarchically hyperbolic.

We also construct CAT(0) structures for a few new examples of $F_3$-by-$\Z$s by thickening the strips in {Bridson}'s tree of spaces construction \cite{BH99}.

\end{abstract}

\section{Introduction} ~

A \emph{free-by-cyclic group} is a semidirect product $F_n\rtimes_\phi\Z$ of a free group $F_n$ of rank $n\ge 1$ with $\Z$. It is determined by a free group automorphism $\phi\in \Aut(F_n)$, and sometimes called the \emph{suspension} of $\phi$. It is likely a difficult open problem to determine which free group automorphisms have a CAT(0) suspension, despite the great amount of machinery available for the study of free group automorphisms.

Significant subclasses of free group automorphisms are known to have CAT(0) suspensions. Free group automorphisms fall into two types: the exponentially-growing ones and the polynomially-growing ones.
All exponentially-growing automorphisms that do not act periodically on a nontrivial conjugacy class define word-hyperbolic free-by-cyclic groups \cite{BF92, Bri00}, and these are cocompactly cubulated thanks to the boundary separation criterion in \cite{BW12, HW15}. Geometric fully irreducible automorphisms, i.e. automorphisms obtained from a pseudo-Anosov homeomorphism of a surface with boundary, give rise to fundamental groups of compact hyperbolic $3$-manifolds with boundary \cite{Ota96} which are known to be CAT(0).

The situation for polynomially-growing automorphisms is less clear. All $F_2$-by-$\Z$s act geometrically on a CAT(0) square complex \cite{BK16}. Most known examples of free-by-cyclic groups that are not CAT(0) are similar to Gersten's example \cite{Ger94}: this is a linearly-growing automorphism $\phi\in \Aut(F_3)$ given as $\phi(a) = a$, $\phi(b) = ba$, $\phi(c) = ca^2$. The first equation implies that the automorphism letter $t$ and the generator $a$ span a $\Z^2$ that should act on an embedded plane as described in the flat torus theorem. The other two equations imply that $t$, $ta=at$ and $ta^2=a^2t$ are all conjugate and should therefore translate by the same length in that plane. So for a point $p$ in the plane, $t\cdot p, (at)\cdot p, (a^2t)\cdot p$ all lie on a circle of the same radius. This is impossible since those same three points also lie on a common translation axis for $a$ giving a circle and a line intersecting in three points in that plane. In fact, this obstruction prevents Gersten's group $F_3\rtimes_\phi\Z$ from even acting properly by semisimple isometries on a CAT(0) space.

We slightly extend the list of non-CAT(0) examples by generalizing Gersten's obstruction as follows.

\begin{thm}
\label{NonCAT0}
    Let $\phi\in \Aut(F_3)$ be given as follows:
    \[\begin{array}{rccl} \phi: &
            a & \mapsto & a \\
            & b & \mapsto & b \\
            & c & \mapsto & cw(a,b)
            \end{array}\]
    with $w\in [F_2, F_2]$, $w\ne 1$.
    
    Then $F_3\rtimes_\phi\Z$ cannot act properly by semisimple isometries on a CAT(0) space; in particular, it is not a CAT(0) group.
\end{thm}

One consequence is the following, which answers question 1.4 in \cite{BGGH26}:

\begin{cor}
\label{NonCAT0HHG}
    Let $\phi\in \Aut(F_3)$ be given as

        \[\begin{array}{rccl} \phi: &
            a & \mapsto & a \\
            & b & \mapsto & ba^k \\
            & c & \mapsto & cw(a,b)
            \end{array}\]
    with $k > 0$ and $w(a,b)\in F_2$ cyclically reduced, $b$-balanced, of $b$-height $h(w) = 1$, starting from its bottom, and totally balanced (see \ref{UPG}).

    Then $F_3\rtimes_\phi\Z$ cannot act properly by semisimple isometries on a CAT(0) space. However, it is a hierarchically hyperbolic free-by-cyclic group.
\end{cor}

Here, the assumptions on $w$ roughly say that $w$ is nontrivial and in the commutator subgroup of the subgroup generated by the fixed elements $a$ and $bab^{-1}$, which imply that the suspension contains a subgroup such as in theorem \ref{NonCAT0}. Examples from the corollary satisfy the ``unbranched blocks'' condition from \cite{BGGH26} that characterizes hierarchically hyperbolic free-by-cyclic groups when the defining automorphism is unipotent polynomially-growing (UPG), as those above are; in fact, they are admissible groups in the sense of \cite{CK02}. A basic example is $a\mapsto a$, $b\mapsto ba$, $c\mapsto c[a, bab^{-1}]$. This contrasts with Gersten's examples, which are ``branched'' and hence not HHGs. Since examples that are both CAT(0) and branched are known \cite{Lym23}, this proves that neither of the classes ``CAT(0) free-by-cyclic groups'' and ``hierarchically hyperbolic free-by-cyclic groups'' is contained in the other. In addition, since our obstruction even prevents proper actions by semisimple isometries on CAT(0) spaces, our examples cannot be properly cubulated, even though they are quasi-isometric to CAT(0) cube complexes.

So far, all known obstructions to being CAT(0) for free-by-cyclic groups come from linearly-growing behaviors, though one can also write exponentially-growing automorphisms whose restriction to some polynomially-growing subgroup has the obstruction. It is therefore relevant to ask when a polynomially-growing automorphism defines a CAT(0) free-by-cyclic group. By twisting Bridson's tree of space construction (\cite{BH99}, II.11), Samuelson constructed CAT(0) $2$-complexes for many polynomially-growing automorphisms \cite{Sam06}, and Lyman used a similar construction for symmetric and palindromic automorphisms \cite{Lym23}.

Using a tree of different spaces, we construct CAT(0) $3$-complexes for new examples with polynomially-growing defining automorphism $\phi\in \Aut(F_3)$.

\begin{thm}
\label{YesCAT0}
    Let $\phi\in \Aut(F_3)$ be of one of the following types:
    \begin{enumerate}
        \item $\begin{array}{rccl} \phi: &
            a & \mapsto & a \\
            & b & \mapsto & b \\
            & c & \mapsto & cw(a,b)
            \end{array}$
            
            with $w\notin [F_2, F_2]$ or $w=1$

        \item $\begin{array}{rccl} \phi: &
            a & \mapsto & a \\
            & b & \mapsto & ba^k \\
            & c & \mapsto & cw(a,b)
            \end{array}$
            
            with $k > 0$ and $w(a,b)\in F_2$ cyclically reduced, $b$-balanced, of $b$-height $h(w) > 0$, and satisfying either $\displaystyle \left(\sum_{j=0}^{h(w)} N_{a,j}(w)\right)^2 < k\sum_{j=0}^{h(w)} (1 + 2(j_0-j))N_{a,j}(w)$ where $j_0\in \{0, \ldots, h(w)\}$ is the starting level of $w$, or $\displaystyle 0 < \sum_{j=0}^{h(w)} (-1)^{j_0-j}N_{a,j}(w) < k$.
 (see \ref{UPG} for notations).
    \end{enumerate}
    Then $F_3\rtimes_\phi \Z$ is a CAT(0) group.
\end{thm}

The first case consists of those automorphisms fixing a free factor of rank $2$ that were left by theorem \ref{NonCAT0}. It seems necessary to reprove that they are CAT(0) since Samuelson's theorem 4.4 \cite{Sam06} had a mistake in its statement and proof. Indeed, it asserts that even those automorphisms as in theorem \ref{NonCAT0} have CAT(0) suspension, which we prove here is not true. The mistake in the proof can be found at the very beginning of page 8: the statement ``decreasing $r$ does not increase this distance'' is not true when $\theta > \frac{\pi}{2}$, which needs to be the case if one wants the ``distance between corners'' to be less than $1$.

The second case consists of automorphisms that do twist both letters $b$ and $c$ nontrivially, with additional conditions on the word $w(a, b)\in F_2$ twisting $c$. Since Samuelson's theorem already deals with the case of a word $w$ having a nonzero sum of powers of $b$, we restrict our attention to those words where this sum is zero (we call them ``$b$-balanced''). The first inequality will guarantee that the letters $a$ move some point in the right way and shorten the displacement of the automorphism letter $t$, as some angular parameter $\theta$ tends to $0$. The second inequality gives additional cases that work when $\theta \longrightarrow \pi$.


Even with these new examples, a number of polynomially-growing automorphisms $\phi\in \Aut(F_3)$ are not known to be CAT(0) or not. It is quite possible that more CAT(0) examples can be constructed through combination theorems with different splittings or semi-splittings, or maybe with yet other variations of Bridson's construction. Completely solving the problem seems hard because it requires dealing with finite index phenomena, but there may be some hope if one restricts to triangular (or, more generally, UPG) automorphisms (see \ref{UPG}). In a follow-up paper \cite{Del26}, we completely solve the problem for a broader class including triangular automorphisms of $F_3$ using a different strategy.

Most of the examples provided by theorem \ref{YesCAT0} do not satisfy the ``rich linearity'' condition from \cite{MP25} that prevents a free-by-cyclic group from acting geometrically on a CAT(0) cube complex, therefore it may be relevant to ask if they are cocompactly cubulated.

In section \ref{Prel}, we begin by introducing notation and terminology to describe our automorphisms of interest (\ref{UPG}). We then review the basic results of CAT(0) geometry that we will use (\ref{CAT0}). Finally, we describe the tree of spaces construction that we will use for the CAT(0) examples (\ref{BHcons}) and we state preliminary lemmas about them (\ref{Thick}). In the next section (\ref{F2}), we prove theorem \ref{NonCAT0} and the first part of theorem \ref{YesCAT0}, thereby correcting the mistake in theorem 4.4 of Samuelson's paper \cite{Sam06}. Finally, we prove the second  part of theorem \ref{YesCAT0} and corollary \ref{NonCAT0HHG} in section \ref{NoF2} (theorem \ref{CAT0ex} and example \ref{nonCAT0ex}). We end with a comparison to the Gersten example, which is not hierarchically hyperbolic.

I would like to thank Peter Samuelson, Hikaru Jitsukawa, Corey Bregman, Robbie Lyman and Mark Hagen for their interest in these results.

\section{Preliminaries} ~
\label{Prel}

\subsection{Triangular automorphisms} ~
\label{UPG}

Let $n\ge 2$. An automorphism $\phi$ of $F_n = F_{\{x_1,\ldots,x_n\}}$ is \emph{triangular} if it is of the form
\[\begin{array}{lcl}
    \phi(x_1) & = & x_1 \\
    \phi(x_2) & = & x_2 w_2(x_1) \\
    & \vdots & \\
    \phi(x_n) & = & x_n w_n(x_1,\ldots,x_{n-1})
\end{array}\]
where $w_i(x_1,\ldots,x_{i-1})\in F_{i-1} = F_{\{x_1,\ldots,x_{i-1}\}}$ (here we call them the \emph{twisting words} of $\phi$). Triangular automorphisms are prototypical examples of \emph{unipotent polynomially-growing (UPG) automorphisms}, i.e. polynomially-growing automorphisms that induce a unipotent matrix of $GL_n(\Z)$ in the abelianization $\Z^n$ of $F_n$. Using (improved relative) train-track theory, it was proven that any polynomially-growing outer automorphism $[\phi]\in \Out(F_n)$ has a nontrivial power $[\phi]^n$ that admits a representative $\widetilde{\phi^n}\in \Aut(F_n)$ that is conjugate to an unipotent-polynomially-growing automorphism $\psi$ \cite{BFH00}. Consequently, the suspensions $F_n\rtimes_{\widetilde{\phi^n}}\Z$ and $F_n\rtimes_{\psi}\Z$ are isomorphic and are finite index subgroups in $F_n\rtimes_\phi \Z$.

Let $\phi\in \Aut(F_n)$ be triangular with twisting words $w_2,\ldots, w_n$. The free-by-cyclic group $F_n\rtimes_\phi\Z$ is presented as
\[F_n\rtimes_\phi \Z = \left\langle x_1,\ldots,x_n, t \left| \begin{array}{rcl}
tx_1t^{-1} & = & x_1 \\
tx_2t^{-1} & = & x_2w_2(x_1) \\
& \vdots & \\
tx_nt^{-1} & = & x_nw_n(x_1,\ldots,x_n)
\end{array}\right .\right\rangle\]
which can be rewritten as an iterated HNN extension
\[F_n\rtimes_\phi \Z = \left\langle x_1,\ldots,x_n, t \left| \begin{array}{rcl}
[t, x_1] & = & 1 \\
x_2^{-1}tx_2 & = & w_2(x_1)t \\
& \vdots & \\
x_n^{-1}tx_n & = & w_n(x_1,\ldots,x_n)t
\end{array}\right .\right\rangle.\]

We focus on the case $n=3$ and use for notation $(a, b, c) = (x_1, x_2, x_3)$. In that case, a triangular automorphism is determined up to conjugacy by $w_1 = a^k$, $k\in \Z$, and $w_2\in F_2 = F_{\{a, b\}}$. Replacing $\phi$ with a conjugate by the automorphism $a\mapsto a^{-1}$, $b\mapsto b$, $c\mapsto c$ if necessary, we may assume $k\ge 0$ without changing $F_3\rtimes_\phi\Z$ up to isomorphism. In first approach, we assume that $w(a,b)$ is cyclically reduced.

For $w = w(a, b)\in F_{\{a, b\}}$, we say that $w$ is \emph{$b$-balanced} if the sum of the powers of $b$ in $w(a, b)$ equals zero.

Assume $w$ is cyclically reduced and $b$-balanced. The \emph{$b$-height} of $w$, denoted $h(w)$, is the maximal sum of powers of $b$ in a cyclic subword of $w$. The \emph{bottom} of $w$ is the set of occurrences of $a$ or $a^{-1}$ where a cyclic subword attaining this maximum begins. If $w\ne 1$, this set is nonempty: indeed, the last occurrence of a power of $b$ (if any) before a cyclic subword attaining the maximum has to be a negative power of $b$, while the first occurrence of a power of $b$ (if any) in that cyclic subword must be a positive power of $b$, so by reduction a nontrivial power of $a$ must separate them. Similarly, the \emph{top} of $w$ is the set of occurrences of $a$ or $a^{-1}$ where a cyclic subword attaining the maximum ends. Any cyclic subword starting in the bottom and ending in the top of $w$ maximizes the sum of powers of $b$.

More generally, for $0\le j\le h(w)$, the \emph{$j$-th level} of $w$ is the set of occurrences of $a$ or $a^{-1}$ where a cyclic subword starting in the bottom and whose sum of powers of $b$ equals $j$ ends (or, equivalently, where a cyclic subword ending in the top and whose sum of powers of $b$ equals $h(w)-j$ begins). We also allow for empty levels, in case no $a$ or $a^{-1}$ occurs between consecutive $b$'s or $b^{-1}$'s. Thus the $0$-th level is the bottom and the $h(w)$-th level is the top. The \emph{starting level} of $w$ is the level entered at the last occurrence of $b$ or $b^{-1}$ and exited at the first occurrence of $b$ or $b^{-1}$. We shall often denote it by $j_0$.

The \emph{$a$-weight of the $j$-th level of $w$}, denoted $N_{a,j}(w)$, is the sum of the powers of $a$ that occur in the $j$-th level. These quantities will appear in theorem \ref{CAT0ex}.

For an example, let $w = ba^2b^{-1}a^{-1}b^{-1}aba^{-2}b^2a^{-1}b^{-1}ab^{-1}a^3$. The levels of $b$ may be represented as follows:
\[\begin{array}{rccccccccccccccc}
\text{Level } 3 \text{ (Top)} & & & & & & & & & & & a^{-1} \\
& & & & & & & & & & b & & b^{-1} & \\
\text{Level } 2 &  & a^2 & & & & & & & & & & & a & \\
& b &  & b^{-1} & & & & & & b & & & & & b^{-1} \\
\text{Level } 1 &  & & & a^{-1} & & & & a^{-2} & & & & & & & a^3 \\
&  & & & & b^{-1} & & b & \\
\text{Level } 0 \text{ (Bottom)} &  & & & & & a &
\end{array}\]
We have $h(w) = 3$, $j_0 = 1$, $N_{a,0}(w) = 1$, $N_{a, 1}(w) = -1 - 2 + 3 = 0$, $N_{a, 2} = 2 + 1 = 3$, and $N_{a, 3} = -1$.

In the tree of spaces construction \ref{BHcons}, the diagram above corresponds to a sequence of adjacent spaces met by a combinatorial path that we shall use to estimate translation lengths. The stacks of $a$'s will correspond to vertex spaces that are euclidean planes and in which these $a$'s translate, and the letters $b$ in-between will translate across the edge spaces binding the planes. The gluing will roughly have the effect of locally turning $b$ into a rotation and $b^{-1}$ into the inverse rotation, thus the translation axes of the $a$'s at a local scale will essentially depend on their level in $w$ only.

\subsection{CAT(0) spaces} ~
\label{CAT0}

We recall here general definitions and facts about CAT(0) spaces and groups. They can be found in \cite{BH99}.

A metric space $(X, d)$ is \emph{geodesic} if any two points in $X$ are joined by a \emph{geodesic}, i.e. an isometrically embedded interval. A \emph{geodesic triangle} in $X$ is a triple $\Delta = ([x, y], [y, z], [z, x])$ with $x, y, z\in X$ and where $[x, y]$, $[y, z]$, $[z, x]$ (the \emph{sides} of $\Delta$) are choices of geodesics between $x$ and $y$, $y$ and $z$, and $z$ and $x$ respectively. The \emph{Euclidean comparison triangle} $\bar{\Delta} = ([\bar{x}, \bar{y}], [\bar{y}, \bar{z}], [\bar{z}, \bar{x}])\subset \mathbb{E}^2$ is the unique (up to isometry) triangle of the euclidean plane such that $d(\bar{x}, \bar{y}) = d(x, y)$, $d(\bar{y}, \bar{z}) = d(y, z)$ and $d(\bar{z}, \bar{x}) = d(z, x)$. Given a point $p$ in one of the three sides of $\Delta$, say $[x, y]$, its \emph{comparison point} $\bar{p}$ is the point of $[\bar{x}, \bar{y}]$ such that $d(\bar{p}, \bar{x}) = d(p, x)$ (and so $d(\bar{p}, \bar{y}) = d(p, y)$).

A metric space $(X, d)$ is \emph{CAT(0)} if it is geodesic and for all geodesic triangle $\Delta$, for any $p, q$ in the union of the sides of $\Delta$, we have $d(p, q) \le d(\bar{p}, \bar{q})$. CAT(0) metric spaces are \emph{uniquely geodesic} (i.e. there is exactly one geodesic between any two points).

A metric space $(X, d)$ is \emph{proper} if its closed balls are compact.

Let $G$ be a discrete group acting on a metric space $X$ by isometries. The action is \emph{(metrically) proper} if for all $x\in X$, there exist $\eps > 0$ such that $\{g\in G | gB(x, \eps)\cap B(x, \eps)\}$ is finite. In particular, the action is properly discontinuous; the two notions coincide if $X$ is proper. It is \emph{cocompact} if there exists a compact $K\subset X$ such that $G\cdot K = X$. A proper cocompact action is called \emph{geometric}. A group $G$ \emph{is a CAT(0) group} if it acts geometrically on some CAT(0) space $(X, d)$. In particular, such a group is finitely generated by the \textsc{Švarc-Milnor} lemma and the CAT(0) metric space $X$ is proper by the metric \textsc{Hopf-Rinow} theorem.

Let $X$ be a CAT(0) space and $\gamma\in \mathrm{Isom}(X)$. The \emph{translation length} of $\gamma$ is the number
\[\|\gamma\| = \inf_{x\in X} d(\gamma(x), x)\]
and the \emph{minimal set} of $\gamma$ is $\Min(\gamma) = \{x\in X | d(\gamma(x), x) = \|\gamma\|\}$. The isometry $\gamma$ is \emph{semisimple} is its minimal set is nonempty (otherwise it is \emph{parabolic}). If $\gamma$ is semisimple, it is called \emph{elliptic} if $\|\gamma\| = 0$ (i.e. $\gamma$ has fixed points) and \emph{hyperbolic} if $\|\gamma\| > 0$. A group $G$ is sometimes called \emph{weakly CAT(0)} if it acts properly by semisimple isometries on a CAT(0) space. A geometric action is always by semisimple isometries.

A subset of the CAT(0) space $X$ is \emph{convex} if any geodesic with endpoints in $X$ lies in $X$. A convex subset of a CAT(0) space is itself CAT(0). If a nonempty closed convex subset $C$ is stabilized by an isometry $\gamma$, then the translation lengths of $\gamma$ and its restriction to $C$ are the same, and $\gamma$ is semisimple if and only if its restriction to $C$ is semisimple.

The minimal set $\Min(\gamma)$ of an isometry $\gamma$ of $X$ is a closed convex subset. If $\gamma$ is hyperbolic, then $\Min(\gamma)$ is isometric to a cartesian product $C\times \R$ (with the $l^2$-metric) where $C$ is a closed convex subset of $X$. Moreover, any isometry $g$ of $X$ that commutes with $\gamma$ stabilizes $\Min(\gamma)$ and its restriction to $\Min(\gamma)$ preserves the product structure, i.e. it is of the form $(g_C, g_\R)$ where $g_C$ is an isometry and $g_\R$ is a translation. If $g$ is semisimple, then so is $g_C$. We have $\gamma_C = id_C$, and $\gamma_\R$ is a translation by $\|\gamma\|$. The fibres $\{c\}\times\R$ of $\Min(\gamma)$ are the \emph{axes} of $\gamma$, i.e. the isometrically embedded real lines in $X$ that are left invariant by $\gamma$; hence $\Min(\gamma)$ is the union of the axes of $\gamma$.

For $\gamma, g\in \mathrm{Isom}(X)$, if $\Min(\gamma)$ is the minimal set of $\gamma$ then we have $\Min(g\gamma g^{-1}) = g\Min(\gamma)$ and $\|g\gamma g^{-1}\| = \|\gamma\|$.



\subsection{Graphs of CAT(0) spaces} ~
\label{BHcons}

We start by recalling the combination theorem from \cite{BH99}, chapter II.11.

Suppose a group $G$ splits over a graph of groups $(\mathcal{G}(V, E, \alpha, \bar{~}), (G_v)_{v\in V}, (G_e)_{e\in E}, (i_{e,\alpha})_{e\in E})$. Suppose $G_v$ acts on a space $X_v$ and $G_e$ acts on a space $X_e$ for all $v\in V$ and $e\in E$ (with $X_e = X_{\bar{e}}$), and that there are $G_e$-invariant subspaces $X_{e, \alpha}\subset X_e$ and $G_e$-equivariant maps $f_{e, \alpha}: X_{e, \alpha}\longrightarrow X_{\alpha(e)}$ for all $e\in E$. Then one may glue the actions together according to the Bass-Serre tree $T$ as follows: take copies $\tilde{X}_{\tilde{v}}$ of $X_v$ for all lift $\tilde{v}\in T$ of $v\in V$ and copies $\tilde{X}_{\tilde{e}}$ of $X_e$ for all lifts $\tilde{e}\subset T$ of $e\in E$ (with only one copy for each pair of opposite edges), and glue every copy $\tilde{x}\in \tilde{X}_{\tilde{e}}$ of a point $x\in X_{e, \alpha}\subset X_e$ to the copy $\tilde{y}$ of $f_{e,\alpha}(x)\in X_{\alpha(e)}$ in $\tilde{X}_{\widetilde{\alpha(e)}}$, such that $\widetilde{\alpha(e)}$ is the initial vertex of $\tilde{e}$ in $T$, and if $g\in \mathrm{Stab}(\widetilde{\alpha(e)})$ then $g\tilde{x}\in \tilde{X}_{g\tilde{e}}$ is glued to $g\tilde{y}$ (see figure \ref{FigExBH}).

\begin{figure}[h!]
    \centering
    \includegraphics[width=0.5\linewidth]{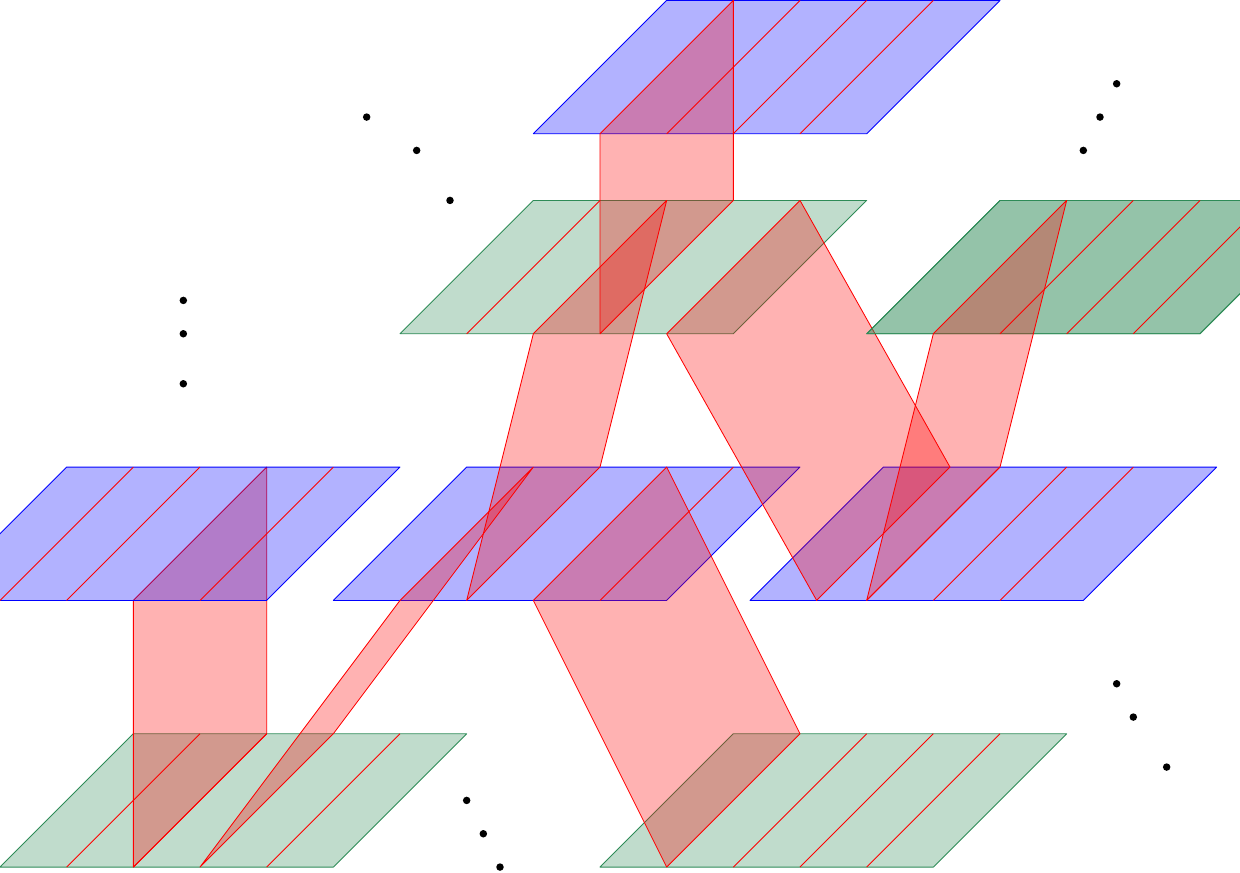}
    \caption{Example of a tree of spaces. Here the two groups are $\Z^2$ acting on planes (blue and green respectively), amalgamated over $\Z$ acting on a strip whose intersection with either vertex space is a line (red). This is the universal cover of a green torus and a blue torus glued to both sides of a red cylinder.}
    \label{FigExBH}
\end{figure}

Let $X$ be the resulting space. Then for all lifts $\tilde{v}\in T$ of $v\in V$, the conjugate of $G_v$ stabilizing $\tilde{v}$ acts on $\tilde{X}_{\tilde{v}}$, and similarly for all $e\in E$ and all lift $\tilde{e}$ of $e$. The actions fit together into an action of $G$ on $X$. If the actions of $G_v$ on $X_v$ and $G_e$ on $X_e$ are cocompact for all $v\in V$, $e\in E$, then the action of $G$ on $X$ is cocompact.

Assume now that the $(X_v, d_v)$ and the $(X_e, d_e)$ are metric spaces and the maps $f_{e, \alpha}$ are isometric embeddings. If the actions of $G_v$ on $X_v$ and $G_e$ on $X_e$ are proper for all $v\in V$ and $e\in E$, and if there is $\eps > 0$ such that $d_e(X_{e, \alpha}, X_{\bar{e}, \alpha}) > \eps$ for all $e\in E$, then the action of $G$ on $X$ is also proper. Finally, assume that for all $v\in V$ and $e\in E$, $X_v$ and $X_e$ are CAT(0), and $X_{e, \alpha}$ is a convex subspace of $X_e$. Then $X$ is CAT(0).

These facts altogether allow to prove, for instance, that a free product of CAT(0) groups is CAT(0) (taking $X_e$ to be segments and $X_{e, \alpha}, X_{\bar{e}, \alpha }$ to be the endpoints of $X_e$, and any map  for $f_{e, \alpha}$), as well as a free amalgamated product of two CAT(0) groups over $\Z$ (taking $X_e$ to be a strip on which $\Z$ acts by translation, and gluing its sides isometrically along axes of the $\Z$ in both groups, after possibly rescaling one of the spaces so that the translation lengths match).

In the case of an HNN extension with CAT(0) vertex group $G_v$ and $G_e\cong \Z$, this construction produces a geometric action of $G$ on a CAT(0) space, provided that there exists a geometric action of $G_v$ on a CAT(0) space $X$ such that the translation lengths of the generators of $i_{e, \alpha}(G_e)$ and $i_{\bar{e}, \alpha}(G_e)$ are the same. This is not always the case: indeed, if $\phi\in \Aut(F_3)$ is given by $\phi(a)=a$, $\phi(b)=ba$, and $\phi(c)=ca^2$, then the construction worked once proves that the $F_2\rtimes\Z$ defined by the restriction of $\phi$ to $F_{\{a, b\}}$ is CAT(0), but one cannot repeat the procedure to obtain a geometric CAT(0) action of $F_3\rtimes_\phi\Z$, for this group is not CAT(0) \cite{Ger94}.

However, Samuelson proved that if $\phi\in \Aut(F_3)$ satisfies $\phi(a)=a$, $\phi(b) = ba^k$ with $k\ne 0$, and $\phi(c) = cw(a, b)$ where the sum of powers of $b$ in $w$ is nonzero, then one can adjust some geometric parameters in the first step of the construction in such a way to have $\|t\| = \|w(a,b)t\|$, so that the second step works \cite{Sam06}. His idea was to bring the axes of $t$ and $w$ ``sufficiently aligned'' by making $\|a\|$ (or, equivalently, the angle $\theta > 0$ between the axes of $t$ and $ta^k$) very small, so that the displacement of a point of the axis of $t$ under $w$ is essentially due to the contributions of $b$; then to introduce a ``twist'' parameter by which $b$ translates as it moves a point across the flat strip between the planes, and to take advantage of the nonzero exponent sum to make the total twist ``cancel'' with $t$ so as to arrange $\|wt\| < \|t\| = 1$; and to conclude by continuously increasing $\|wt\|$ back to $1$ by enlarging the width of the strips.

\subsection{The thickened construction} ~~
\label{Thick}

We shall thus be interested in the remaining cases where the sum of powers of $b$ in $w$ is zero. In this situation, Samuelson's twist parameter will not be of any help, because the introduced translations will tend to cancel altogether. Instead, we will regard $b$ as having a ``rotating'' effect on the axes in the flat, and show that in some cases these rotations move a point in the right way. We achieve this by replacing the strips in the construction with a product of a thin rectangle with a line.

For $k \ge 0$, let $\phi_k$ be the automorphism of $F_2$ such that $\phi_k(a)=a$ and $\phi_k(b)=ba^k$. We have the HNN splitting $F_2\rtimes_{\phi_k}\Z \cong \Z^2 \bigast_{b^{-1}tb=ta^k}$, where $\Z^2 = \langle a, t\rangle$. We let $X_v = \R^2$ be a plane on which $G_v = \Z^2$ acts by translations. We want a metric such that $\|t\| = \|ta^k\| = 1$. If $k=0$, under this condition the metric is determined by the choice of $\|a\| > 0$ and the angle $\varphi\in (0, \pi)$ between the axes of $t$ and $a$. If $k>0$, it is determined by the angle $\theta\in (0, \pi)$ between the axes of $t$ and $ta^k$; we then have $\|a\| = \frac{2}{k}\sin(\frac{\theta}{2})$ (see figure \ref{FigAxesInPlane}). We choose orthonormal coordinates of $\R^2$ such that $t$ translates by the vector $\begin{pmatrix} 1 \\ 0 \end{pmatrix}$ and $ta^k$ by $\begin{pmatrix} \cos(\theta) \\ \sin(\theta) \end{pmatrix}$. We then let $\eps > 0$ (the ``distance between adjacent flats'') and $D > 0$ (the ``thickness'' of the strips) and set $X_e = \R\times [-D, D]\times [0, \eps]$, $X_{e,\alpha} = \R\times[-D, D]\times\{0\}$ and $X_{\bar{e}} = \R\times[-D, D]\times\{\eps\}$. Finally, for $(u, v)\in \R\times[-D, D]$, we set $f_{e, \alpha}(u, v, 0) = (u, v)$ and $f_{\bar{e}, \alpha}(u, v, \eps) = (\cos(\theta)u-\sin(\theta)v, \sin(\theta)u + \cos(\theta) v)$ (see figure \ref{FigGluing}).

\begin{figure}[h!]
    \centering
    \includegraphics[width=0.5\linewidth]{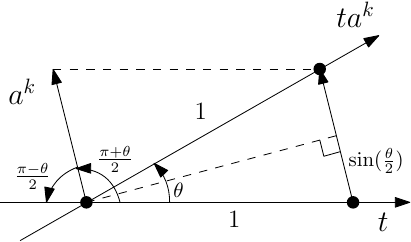}
    \caption{The axes of $t$ and $ta^k$ in the flat plane, $k > 0$}
    \label{FigAxesInPlane}
\end{figure}

\begin{figure}[h!]
    \centering
    \includegraphics[width=0.8\linewidth]{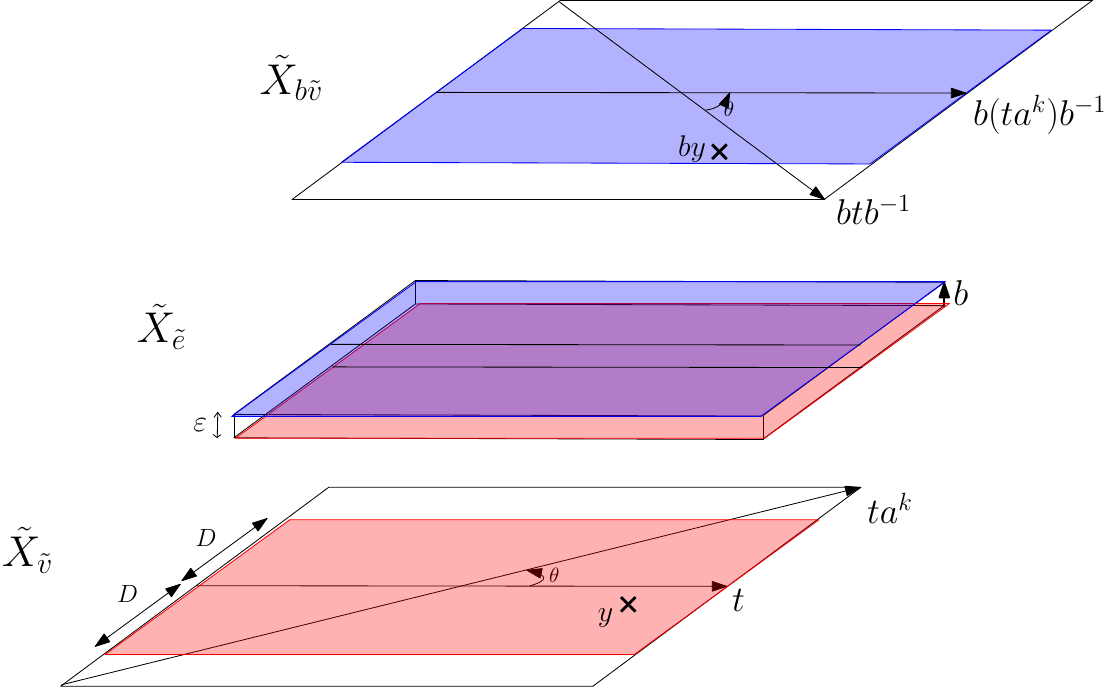}
    \caption{The gluing between a plane $\tilde{X}_{\tilde{v}}$ and its image under $b$, with the edge space $\tilde{X}_{\tilde{e}}$ joining them.}
    \label{FigGluing}
\end{figure}

If $k > 0$, we denote by $X_k(\theta, D, \eps)$ the space obtained from the gluing construction, and $X_0(\varphi, \|a\|, D, \eps)$ if $k=0$. This is a CAT(0) space on which $F_2\rtimes_{\phi_k}\Z$ acts geometrically. As in Samuelson's theorem, we need a lemma for a future intermediate value argument:

\begin{lem}
\label{Cont}
    Let $D > 0 $. In both of these cases:
    \begin{enumerate}
        \item Let $k > 0$ and $\theta\in (0, \pi)$. Let $g\in F_2\rtimes_{\phi_k}\Z$, and let $f_g: \R_+^* \longrightarrow \R_+$ be such that $f_g(\eps)$ is the translation length of $g$ in $X_k(\theta, D, \eps)$.
        \item Let $\varphi\in (0, \pi)$ and $\|a\| > 0$. Let $g\in F_2\rtimes_{\phi_0}\Z$, and let $f_g: \R_+^* \longrightarrow \R_+$ be such that $f_g(\eps)$ is the translation length of $g$ in $X_0(\varphi, \|a\|, D, \eps)$.
    \end{enumerate}
     Then the translation length function $f_g$ of $g$ is a continuous function of the distance $\eps$ between adjacent planes.
\end{lem}

\begin{proof} ~

Let $\eps > 0$. For $\eps'\in (-\eps, \eps)$, there is a natural, $F_2\rtimes_{\phi_k}\Z$-equivariant, piecewise-linear homeomorphism $I: X_k(\theta, D, \eps) \longrightarrow X_k(\theta, D, \eps + \eps')$ (resp. $X_0(\varphi, \|a\|, D, \eps) \longrightarrow X_0(\varphi, \|a\|, D, \eps + \eps')$) obtained by piecing together the linear maps $\R\times [-D, D]\times [0, \eps] \longrightarrow \R\times [-D, D]\times [0, \eps + \eps']$ and the identity maps on the planes. The map $I$ restricted to the edge spaces $X_{e}$ is $\Lambda = \max(1+\frac{\eps'}{\eps}, \frac{1}{1+\frac{\eps'}{\eps}})$-bilipschitz, hence so is $I$ by definition of a length pseudo-metric obtained by gluing length spaces. If $g$ is the identity, then $f$ is constant equal to $0$ and therefore continuous at $\eps$. Otherwise (since $F_2\rtimes_{\phi_k}\Z$ is torsion-free and the actions are by semisimple isometries because of cocompactness) we have $\|g\| > 0$ on either space and
\[0 < \frac{1}{\Lambda}f(\eps) \le f(\eps + \eps') \le \Lambda f(\eps),\]
which proves continuity at $\eps$ since $\Lambda \longrightarrow 1$ as $\eps' \longrightarrow 0$.

\end{proof}

We now formulate the key lemmas for proving that in certain situations, we can arrange $\|wt\| < \|t\|=1$.

\begin{lem}
\label{FlatkPos}
    Let $k > 0$ and $w\in F_2$ be cyclically reduced and $b$-balanced of length $N$, and let $j_0\in \{0, \ldots, h(w)\}$ be the starting level of $w$. Let $\theta\in (0, \pi)$, $D \ge 2N$ and $\eps > 0$. Let $\rho_\theta$ be the rotation of the plane $X_v$ by the angle $\theta$ in the coordinates above, and let $\vec{a}$ be the vector in $\R^2$ by which $a$ translates in these coordinates. Let $x_0$ be the point of $X_v$ with coordinates $(-1, 0)$, and let $\tilde{x}_0\in X_k(\theta, D, \eps)$ be its lift in the lift $\tilde{X}_{\tilde{v}}$ of $X_v$ stabilized by $G_v = \langle a, t\rangle$. Let $x_1\in X_v$ be the point whose coordinates are $\displaystyle \sum_{j=0}^{h(w)} N_{a,j}(w)\rho_\theta^{j_0-j}\vec{a}$, and let $\tilde{x}_1$ be its lift in $\tilde{X}_{\tilde{v}}$. Then $d(wt\tilde{x}_0, \tilde{x}_1) \le N\eps$.
\end{lem}

\begin{proof} ~

For $l\in \{0, \ldots, N\}$, let $w_l$ be the suffix of $w$ of length $l$. Let $y_0 = t\tilde{x}_0$ be the origin of $\tilde{X}_{\tilde{v}}$. Let $\tau: F_2\longrightarrow \mathrm{Isom}(\tilde{X}_{\tilde{v}}) = \mathrm{Isom}(\R^2)$ be the group homomorphism sending $a$ to the translation by $\vec{a}$ and $b$ to the rotation $\rho_\theta^{-1}$.

We now see by induction that for all $l\in \{0, \ldots, N\}$, we have $d(w_ly_0, \tau(w_l)y_0) \le l\eps$ and $d(\tau(w_l)y_0, y_0) \le 2l$. This is clear for $l=0$. To go from $l$ to $l+1$, first consider the case where the first letter of $w_{l+1}$ is $a^\nu$, $\nu\in \{-1, 1\}$. Since $\tau(w_l)y_0\in \tilde{X}_{\tilde{v}}$, we have $a^\nu\tau(w_l)y_0 = \tau(w_l)y_0 + \nu \vec{a} = \tau(a^\nu)\tau(w_ly_0) = \tau(w_{l+1}y_0)$, hence 
\[d(w_{l+1}y_0, \tau(w_{l+1})y_0) = d(a^\nu w_ly_0, a^\nu \tau(w_l)y_0) = d(w_ly_0, \tau(w_l)y_0) \le (l+1)\eps\]
and $d(\tau(w_{l+1})y_0, y_0) \le d(\tau(w_l)y_0, y_0) + \|a\| \le 2l+2$, since $\|a\| = \frac{2}{k}\sin(\frac{\theta}{2}) \le 2$.

Now assume that the first letter of $w_{l+1}$ is $b^\nu$, $\nu\in\{-1, 1\}$. Then $b^\nu \tau(w_l)y_0\in \tilde{X}_{b^\nu\tilde{v}}$ is a distance at most $2l \le D$ from $b^\nu y_0$, so it is in the lift $\tilde{X}_{\tilde{e}}$ of $X_e$ connecting $\tilde{X}_{\tilde{v}}$ to $\tilde{X}_{b^\nu\tilde{v}}$. According to the gluing, the matching point in $\tilde{X}_{\tilde{v}}$ on the opposite side of $\tilde{X}_{\tilde{e}}$ is $\rho_\theta^{-\nu}(b)\tau(w_l)y_0 = \tau(w_{l+1})y_0$ (see figure \ref{FigGluing}), thus
\[d(w_{l+1}y_0, \tau(w_{l+1})y_0) \le d(w_ly_0, \tau(w_l)y_0) + d(b^\nu\tau(w_l)y_0, \tau(w_{l+1})y_0) \le l\eps + \eps.\]
Also, since $\rho_\theta$ fixes $y_0$, we have $d(\tau(w_{l+1})y_0, y_0) = d(\tau(w_l)y_0, y_0) \le 2l$. This concludes the induction.

It only remains to check that $\tilde{x}_1 = \tau(w)y_0$. Indeed, $\tau(w)$ is the sum over the occurrences of $a^\nu$ ($\nu\in \{-1, 1\}$) of $\nu \rho_\theta^{-n}\vec{a}$, where $n$ is the sum of powers of the $b$'s that occur before said occurrence of $a^\nu$ in $w$. We have $n = j-j_0$ if the occurrence of $a^\nu$ is in the $j$-th level, and regrouping the terms gives the desired formula.

\end{proof}

We have a similar statement for $k=0$:

\begin{lem}
\label{Flatk0}
    Let $w\in F_2$ be a word of length $N$, and let $N_a\in \Z$ be the sum of powers of $a$ in $w$. Let $\varphi\in (0, \pi)$, $\|a\| > 0$, $D \ge N\|a\|$ and $\eps > 0$, and let $\vec{a}$ be the vector by which $a$ translates in the above coordinates for $X_v$. Let $x_0\in X_v$ have coordinate $(-1, 0)$, $x_1\in X_v$ have coordinates $N_a\vec{a}$, and let $\tilde{x}_0\in X_k(\varphi, \|a\|, D, \eps)$ and $\tilde{x}_1\in X_k(\varphi, \|a\|, D, \eps)$ be their lifts in the lift $\tilde{X}_{\tilde{v}}$ of $X_v$ stabilized by $G_v=\langle a, t\rangle$. Then $d(wt\tilde{x}_0, \tilde{x}_1) \le N\eps$.
\end{lem}

\begin{proof} ~

The proof is similar to that of lemma \ref{FlatkPos}, replacing $\tau(b) = \rho_\theta^{-1}$ with $\tau(b) = id$ and the second inequality in the induction by $d(\tau(w_l)y_0, y_0) \le l\|a\|$. Since $\rho_\theta$ is replaced by the identity, the coordinates for $\tau(w)y_0$ are simply the sum over all occurrences of $a^\nu$ ($\nu\in \{-1, 1\}$) of $\nu\vec{a}$. This sum is $N_a\vec{a}$.

\end{proof}

\section{Automorphisms fixing two generators} ~
\label{F2}

We are now ready to prove the corrected statement of Samuelson's theorem 4.4.

\begin{pro}
\label{InvF2}
    Let $\phi$ be the automorphism of $F_3 = F_{\{a, b, c\}}$ such that $\phi(a)=a$, $\phi(b)=b$, and $\phi(c)=cw(a, b)$ where $w(a, b)\in F_{\{a, b\}}$. 
    \begin{enumerate}
        \item If $w(a, b)=1$ or $w(a, b)\notin [F_{\{a,b\}}, F_{\{a,b\}}]$, then $F_3\rtimes_\phi\Z$ is CAT(0).
        \item Otherwise, $F_3\rtimes_\phi\Z$ cannot act properly by semisimple isometries on a CAT(0) space.
    \end{enumerate}
\end{pro}

\begin{proof} ~

We first prove 1. If $w(a, b)=1$, then $F_3\rtimes_\phi\Z = F_3\times \Z$ is clearly CAT(0). Assume $w(a, b)\notin [F_{\{a, b\}}, F_{\{a, b\}}]$. This means that either the sum of powers of $a$ or the sum of powers of $b$ in $w$ is nonzero. Up to a change of basis, we may assume $N_a \ne 0$. In that situation, clearly we may choose $\varphi\in (0, \pi)$, $\|a\| > 0$ such that $N_a\vec{a}$ is in the open euclidean disk of center $(-1, 0)$ and radius $1$; for example, one may choose $\vec{a} = \frac{1}{2N_a}\begin{pmatrix} 1 \\ 1 \end{pmatrix}$. Therefore, for $\eps < \frac{1-d(\vec{a}, (-1, 0))}{N}$ and $D\ge 2\|a\|$, lemma \ref{Flatk0} proves that we have $\|wt\| < 1$ in $X_0(\varphi, \|a\|, D, \eps)$. If $b$ appears in $w$, we find values of the parameters such that $\|wt\| = 1 = \|t\|$ by increasing $\eps$ continuously and using lemma \ref{Cont}. Otherwise, since the axis of $wt$ is in the plane $\tilde{X}_{\tilde{v}}$, we can choose $\vec{a}$ such that $\|wt\|=1$ in a way similar to the case where $k > 0$. In either case, the tree of spaces construction (with strips as edge spaces) works a second time, proving that $F_3\rtimes_\phi\Z$ is CAT(0).

We now prove 2. Assume $1\ne w(a, b)\in [F_{\{a, b\}}, F_{\{a, b\}}]$ and suppose for a contradiction that $F_3\rtimes_\phi\Z$ acts properly by semisimple isometries on a CAT(0) space $X$. We may assume by rescaling that the translation length of the automorphism letter $t$ is $\|t\|=1$. The minimal set $\Min(t)$ of $t$ splits isometrically as a product $C\times \R$, where $C\subset X$ is a closed convex subset of $X$, such that $t$ acts on $\Min(t)$ as $\mathrm{id}_C\times 1$ through this splitting. Now, any element $x\in F_{\{a, b\}}$ commutes with $t$, and therefore leaves $\Min(t)$ invariant and acts as a product isometry $x_C\times x_\R$ in the coordinates of the splitting, where $x_\R$ is a translation of $\R$ by some length $x_\R$. The map $x\mapsto x_\R$ is a group homomorphism into the abelian group $\R$, so since $w(a, b)$ is in the commutator subgroup of $F_{\{a, b\}}$, we have $w(a,b)_\R = 0$ and $w(a,b)$ leaves $C\times\{0\}$ invariant. Since $w\ne 1$ is semisimple, it has an axis $\Delta$ in $C$ on which it translates by some number $\|w(a,b)\| > 0$. Therefore, $w(a, b)t$ has an axis in the flat plane $\Delta\times \R$, on which its translation length is given by the Pythagorean theorem as
\[\|w(a, b)t\| = \sqrt{\|w(a, b)\|^2 + 1} > 1.\]
This contradicts the equality $c^{-1}tc = w(a, b)t$, for conjugate elements must have same translation length.

\end{proof}

\section{Examples with $k>0$} ~
\label{NoF2}

We now turn our attention to case 2. in theorem \ref{YesCAT0}.

\begin{thm}
\label{CAT0ex}
    Let $k > 0$ and $w\in F_{\{a, b\}}$ be a cyclically reduced $b$-balanced word not in $\langle a\rangle$. Let $\phi$ be the automorphism of $F_3=F_{\{a,b,c\}}$ given by $\phi(a)=a$, $\phi(b)=ba^k$, and $\phi(c) = cw(a,b)$. If one of the following conditions holds:
    \begin{enumerate}
        \item We have
        \[\left(\sum_{j=0}^{h(w)} N_{a,j}(w)\right)^2 < k\sum_{j=0}^{h(w)} (1 + 2(j_0-j))N_{a,j}(w),\]
        where $j_0$ is the starting level of $w$;
        \item We have
        \[0 < \sum_{j=0}^{h(w)} (-1)^{j_0-j}N_{a,j}(w) < k\]
    \end{enumerate}
    Then $F_3\rtimes_\phi\Z$ is CAT(0).
\end{thm}

\begin{proof} ~

Lemma \ref{FlatkPos} proves that if there exists $\theta\in (0, \pi)$ such that $\sum_{j=0}^{h(w)}N_{a,j}(w)\rho_\theta^{j_0-j}\vec{a}$ is in the open euclidean disk of radius $1$, then for $D \ge 2N$ (where $N$ is the length of $w$) and a distance $\eps > 0$ between adjacent flats small enough, we have $\|wt\| < 1$, and lemma \ref{Cont} proves that continuously increasing $\eps$ allows to find values of the parameters for which $\|wt\| = 1 = \|t\|$ (here we are using that $b$ appears in the cyclically reduced word $w$ to ensure that $\|wt\|\longrightarrow +\infty$ as $\eps\longrightarrow +\infty$), so that the tree of spaces construction works a second time (with strips as edge spaces) to produce a geometric action of $F_3\rtimes_\phi\Z$ on a CAT(0) space. Thus, it suffices to show that either condition 1 or 2 imply the existence of such a parameter $\theta\in (0, \pi)$.


We write coordinates in $\R^2$ with complex numbers for ease of computation. We thus have $\vec{a} = \frac{2}{k}\sin(\frac{\theta}{2})e^{i\frac{\pi+\theta}{2}}$, and $\rho_\theta$ acts as multiplication by $e^{i\theta}$, and we want to find $\theta\in (0, \pi)$ such that, for some $0\le j_0\le h(w)$,
\[\delta(\theta) = \left|1 + \sum_{j=0}^{h(w)} N_{a,j}(w)ie^{i(j_0-j + \frac{1}{2})\theta}\frac{2}{k}\sin\left(\frac{\theta}{2}\right)\right|^2 < 1.\]

We will find conditions 1 and 2 using series expansions near $\theta=0$ and $\theta=\pi$, respectively. For $\theta$ near $0$, we find
\[1 + \sum_{j=0}^{h(w)} N_{a,j}(w)ie^{(j_0-j+\frac{1}{2})\theta}\frac{2}{k}\sin\left(\frac{\theta}{2}\right) = 1 + \frac{i}{k}\sum_{j=0}^{h(w)}N_{a,j}(w) \theta - \sum_{j=0}^{h(w)}\frac{j_0-j+\frac{1}{2}}{k}N_{a,j}(w)\theta^2 + O(\theta^3),\]
then
\[\delta(\theta) = 1 + \left(\frac{1}{k^2}\left(\sum_{j=0}^{h(w)} N_{a,j}(w)\right)^2 - \sum_{j=0}^{h(w)}\frac{2(j_0-j)+1}{k}N_{a,j}(w)\right)\theta^2 + O(\theta^3).\]
For $\theta$ close to $0$, this is less than $1$ provided that the quadratic term is negative, which gives condition 1.

For $\theta$ near $\pi$, we get
\[1 + \sum_{j=0}^{h(w)}N_{a,j}(w)ie^{i(j_0-j+\frac{1}{2})\theta}\frac{2}{k}\sin\left(\frac{\theta}{2}\right) \underset{\theta\longrightarrow \pi}{\longrightarrow} 1 - \frac{2}{k}\sum_{j=0}^{h(w)}(-1)^{j_0-j}N_{a,j}(w),\]
and the limit of $\delta(\theta)$ as $\theta\longrightarrow \pi$ is less than $1$ if $0 < \frac{2}{k}\sum_{j=0}^{h(w)}(-1)^{j_0-j}N_{a,j}(w) < 2$. 
This gives condition 2.

\end{proof}

\begin{req}
    In the special case where $w\in [F_2, F_2]$, i.e. $N_a(w) = \sum_{j=0}^{h(w)} N_{a,j}(w) = 0$, the first condition becomes $\sum_{j=0}^{h(w)}jN_{a,j}(w) < 0$. This can clearly be achieved, for example with $w = [a, b]$, so we do get new examples indeed. Cases where both conditions fail include the cases where all $N_{a,j}$'s are zero.
\end{req}

We finally prove Corollary \ref{NonCAT0HHG}. Let $w(a, b)\in F_2$ be cyclically reduced and $b$-balanced, and let $0\le j\le h(w)$. We say that the $j$-th level of $w$ is \emph{balanced} if $N_{a,j} = 0$. We say that $w$ is \emph{totally balanced} if all its levels are balanced.

\begin{exm}
\label{nonCAT0ex}
    Let $k > 0$ and $w\in F_{\{a, b\}}$ be cyclically reduced, $b$-balanced, starting from its bottom, totally balanced and such that $h(w)=1$. Let $\phi\in \Aut(F_{\{a, b, c\}})$ be given by $\phi(a)=a$, $\phi(b)=ba^k$, $\phi(c) = cw(a,b)$. Then $F_3\rtimes_\phi\Z$ cannot act properly by semisimple isometries on a CAT(0) space.

    However, $F_3\rtimes_\phi \Z$ is hierarchically hyperbolic, which for such free-by-cyclic groups is equivalent to being quasi-isometric to a CAT(0) cube complex \cite{BGGH26}.
\end{exm}

\begin{proof} ~

We first show that the group cannot act properly by semisimple isometries on a CAT(0) space. The elements $\alpha = a$ and $\beta = bab^{-1}$ are fixed by $\phi$ and generate a free subgroup of rank $2$ in $F_2$, and the assumptions say that $w$ is in the commutator subgroup of that subgroup. Therefore, $F_3\rtimes_\phi\Z$ contains a subgroup of the type described in proposition \ref{InvF2} and hence cannot act properly by semisimple isometries on a CAT(0) space.

Let $w'$ be a primitive root of $w$ in $F_2$ (i.e. a generator of the centralizer of $w$, since $w$ is not trivial), and let $\gamma = cw'c^{-1}$. Then $\alpha$, $\beta$ and $\gamma$ generate a free subgroup of rank $3$ that is fixed by $\phi$, that we denote $F$. Since $w$ is cyclically reduced, $w'$ is a prefix of $w$ and there are no cancellations in the equality $w'^n = w$ ($n\ge 1$), so $w'$ is also in the commutator subgroup of the subgroup generated by $\alpha$ and $\beta$ and there exists nontrivial words $\bar{w}'$ and $\bar{w}$ such that $w'(a, b) = \bar{w}'(\alpha, \beta)$ and $w(a, b) = w'(a, b)^n = \bar{w}'(\alpha, \beta)^n = \bar{w}(\alpha, \beta)$.

From there, we could check by hand that $F$ is all of the fixed subgroup of $\phi$, and that this is the only conjugacy class of a maximal subgroup of $F_3$ fixed by a representative of the outer automorphism $[\phi]$. It would follow that the subgroup $F\times \Z \cong F\rtimes_\phi \Z \le F_3\rtimes_\phi \Z$ is a block, and that the blocks of $F_3\rtimes_\phi \Z$ are precisely its conjugates. We could then construct the linear splitting and check that it is unbranched in the sense of \cite{BGGH26}. Instead of doing that, we are going to start constructing the linear splitting directly from the presentation, then show that the group is admissible in the sense of \cite{CK02} and appeal to a weaker result that directly shows hierarchical hyperbolicity. We have:

\[\begin{array}{rcl}

F_3\rtimes_\phi\Z & = & \left\langle \begin{array}{c} a, b, c, \\ \alpha, \beta, \gamma, \\ t \end{array} \left| \begin{array}{c} [t, \alpha] = [t, \beta] = [t, \gamma] = 1, \\ a = \alpha, bab^{-1} = \beta, cw'(a, b)c^{-1} = \gamma, \\ {[} t, a{]} = 1, b^{-1}tb = a^kt, c^{-1}tc = w(a, b)t \end{array} \right .\right\rangle
\\
& = & \left\langle \begin{array}{c} \alpha, \beta, \gamma, t \\ b, c \end{array} \left| \begin{array}{c} [t, \alpha] = [t, \beta] = [t, \gamma] = 1, \\ b^{-1}\beta b = \alpha, ~~ b^{-1}tb = \alpha^kt, \\ c^{-1}\gamma c = \bar{w}'(\alpha, \beta), ~~ c^{-1}tc = \bar{w}(\alpha, \beta)t \end{array} \right .\right\rangle

\end{array}\]

This shows that $F_3\rtimes_\phi \Z$ splits as a graph of groups with one vertex labeled by $F\times\langle t\rangle$ and two edges labeled by $\Z^2$, where the edge monomorphisms map the generators of $\Z^2$ to $(\alpha, \alpha^kt)$ and $(\beta, t)$ on the $b$-edge, and to $(\bar{w}'(\alpha, \beta), \bar{w}(\alpha, \beta)t)$ and $(\gamma, t)$ on the $c$-edge. We shall see that this splitting makes the group admissible, and conclude to hierarchical hyperbolicity using theorem 1.4 in \cite{HRSS25}.

The images of the edge groups are $\langle \alpha\rangle \oplus \langle t \rangle$, $\langle \beta\rangle \oplus \langle t\rangle$, $\langle \bar{w}'(\alpha, \beta)\rangle \oplus \langle t\rangle$, and $\langle \gamma\rangle\oplus \langle t\rangle$. Since $\bar{w}'(\alpha, \beta) = w'(a, b)\in [F_{a,b}, F_{a,b}]$ does not commute with $\alpha = a$, $\beta = bab^{-1}$ or $\gamma = cw'(a, b)c^{-1}$, the pairwise intersections of the images of the edge groups are all equal to the center $\langle t\rangle$ of $F$, so they are not commensurable with $\Z^2$.

Moreover, the subgroups of $\Z^2$ generated by the preimages of the center of $F$ under the edge monomorphisms are $(-k, 1)\Z \oplus (0, 1)\Z$ in the $a$-edge and $(-n, 1)\Z \oplus (0, 1)\Z$ in the $b$-edge, both of which are finite index since $k \ne 0$ and $w\ne 1$ (so $\bar{w}'\ne 1$ either). This finishes to prove admissibility, hence hierarchical hyperbolicity.

\end{proof}

As a final remark, let us see why the Gersten example $a\mapsto a$, $b\mapsto ba$, $c\mapsto ca^2$ is, by contrast, not hierarchically hyperbolic. If we attempt to construct the linear splitting with the same approach, we find generators of the fixed subgroup $F$ given as $\alpha = a$, $\beta = bab^{-1}$ and $\gamma = cac^{-1}$. The resulting presentation writes
\[
F_3\rtimes_\phi\Z = \left\langle \begin{array}{c} \alpha, \beta, \gamma, t \\ b, c \end{array} \left| \begin{array}{c} [t, \alpha] = [t, \beta] = [t, \gamma] = 1, \\ b^{-1}\beta b = \alpha, ~~ b^{-1}tb = \alpha t, \\ c^{-1}\gamma c = \alpha, ~~ c^{-1}tc = \alpha^2t \end{array}\right.\right\rangle
\]
and we can read a splitting over a graph of groups with one vertex carrying $F\times \langle t\rangle$ and two edges labeled $b$ and $c$, where the edge monomorphisms carry the generators of $\Z^2$ to $(\beta, t)$ and $(\alpha, \alpha t)$ on the $b$-edge and to $(\gamma, t)$ and $(\alpha, \alpha^2 t)$ on the $c$-edge.

In this case, admissibility of this splitting fails because both the $b$-edge and the $c$-edge map $\Z^2$ to $\langle \alpha\rangle \oplus \langle t \rangle$. This means that some additional operations will be required to obtain the linear splitting. In general, the linear splitting is a bipartite graph of groups, where the vertices of one color (called ``white'' in \cite{BGGH26}) carry the blocks $F\times \Z$ arising from fixed subgroups and the other vertices (the ``black'' vertices) are labeled with $\Z^2$.

Here, we obtain the linear splitting by subdividing the graph of group of the splitting above, then folding the resulting two half-edges that map to the same subgroup in $F\times \langle t\rangle$. In other words, we introduce a new ``black'' vertex in the middle of both edges carrying $\Z^2$, which we identify with the edge group of both resulting half-edges. Then we quotient the graph by identifying the two half-edges that were on the positive side of the $b$ and $c$ edges, and we obtain a graph with one ``white'' vertex labeled $F\times \langle t\rangle$, one ``black'' vertex labeled $\Z^2$, and three edges $b_-$, $c_-$, and $b_+=c_+$ carrying $\Z^2$. The edge monomorphisms towards the black vertex may be chosen to be the identity on $b_-$ and $b_+$, and the isomorphism of $\Z^2$ sending $((1, 0), (2, 1))$ to $((1, 0), (1, 1))$ on $c_-$ (accounting for the required change of basis in the $c$-edge prior to folding, so that the two folded half-edges $b_+$ and $c_+$ carry the same monomorphism). See figure \ref{FigGoGGer} below.

With this new splitting, no more edge groups intersect in more than the center $\langle t\rangle$ in the white vertex; this is the linear splitting used in \cite{BGGH26}. The word ``branched'' refers to the black vertex having degree greater than 2, and is the criterion for non hierarchical hyperbolicity. In the admissible examples earlier, the linear splitting would simply be the subdivided admissible graph as shown in figure \ref{FigGoGHHG} below. All black vertices have degree 2, so the group is unbranched.

\begin{figure}[h!]
    \centering
    \begin{subfigure}[b]{0.45\textwidth}
         \centering
         \includegraphics[width=\textwidth]{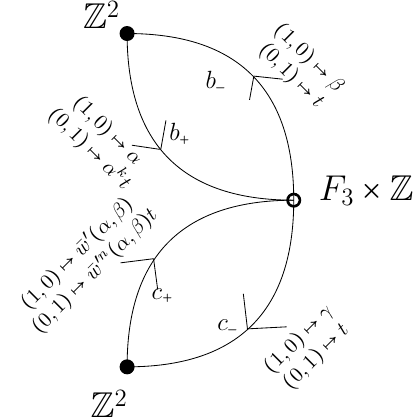}
         \caption{An unbranched non-CAT(0) example}
         \label{FigGoGHHG}
     \end{subfigure}
     \hfill
     \begin{subfigure}[b]{0.45\textwidth}
         \centering
         \includegraphics[width=\textwidth]{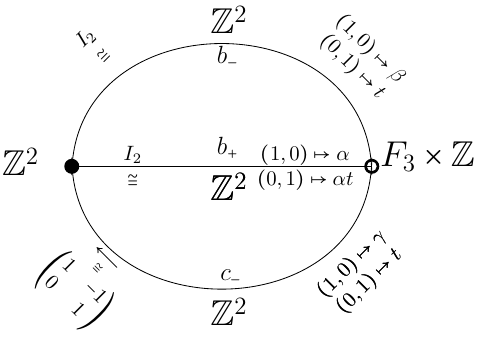}
         \caption{Gersten's example}
         \label{FigGoGGer}
     \end{subfigure}
     \caption{Two linear splittings of non-CAT(0) $F_3$-by-$\Z$s, one unbranched (as in example \ref{nonCAT0ex}) and one branched (Gersten's example).}
    \label{FigGoG}
\end{figure}

\printbibliography

\end{document}